\documentclass[12pt]{article}
\usepackage{amssymb}
\usepackage{amsmath}
\usepackage{amsfonts}
\usepackage{latexsym}
\usepackage{hyperref}
\usepackage{enumerate}
\usepackage{pspicture,epsfig,pstricks}
\usepackage{amsthm,cite}
\usepackage{graphicx,mathtools}
\usepackage{float}
\usepackage[margin=1in]{geometry}
\usepackage{multirow,booktabs} 
\everymath{\displaystyle}
\DeclarePairedDelimiter\ceil{\lceil}{\rceil}

\usepackage{tikz}
\usetikzlibrary{graphs}
\usepackage[mathscr]{euscript}
\title{On the number  of geodesics of Petersen graph $GP(n,2)$ }
\author{\small {Sunil~Kumar R\thanks{ The work of the author is supported by the University Grants Commission (UGC),
Government of India, under the scheme of Faculty Development Programme (FDP) for colleges.}\;\;, Kannan Balakrishnan$^+$}\\
        \small{Department of Computer~Applications}\\
        \small{Cochin University of Science and Technology}\\
        \small{ e-mail:\,$^*$sunilstands@gmail.com, $^+$mullayilkannan@gmail.com} }
           \date{ }

\newtheorem{theorem}{Theorem}[section]
\newtheorem{lemma}{Lemma}[section]

\newtheorem{corollary}{Corollary}[section]
\newtheorem{proposition}{Proposition}[section]
\newtheorem{definition}{Definition}

\begin{document}
\maketitle
\begin{abstract} 
In any network, the interconnection of nodes by means of geodesics and the number of geodesics existing between nodes are important. There exists a class of centrality measures  based on the number of geodesics passing through a vertex. Betweenness centrality indicates the betweenness of a vertex or how often a vertex appears on geodesics between  other vertices. It has wide applications in the analysis of networks. Consider $GP(n,k)$.
For each $n$ and $k \,(n > 2k)$, the generalized Petersen graph $GP ( n , k )$ is a trivalent graph with vertex set
$\{ u_ i ,\, v_ i \,|\, 0 \leq i \leq n - 1 \}$  and edge set $\{ u_ i u_ {i + 1} , u_ i v _i , v_ i v_{ i + k}\, |\, 0\leq i \leq n - 1, \hbox{ subscripts reduced modulo } n \}$. There are three kinds of edges namely outer edges, spokes and inner edges. The outer vertices generate an $n$-cycle called outer cycle and inner vertices generate one or more inner cycles.   In this paper, we consider $GP(n,2)$ and find expressions for the number of  geodesics  and betweenness centrality.\\

\textbf{Keywords}: Petersen graph, geodesics, wicket,  M\"{o}bius strip, betweenness centrality,  induced betweenness centrality.
 \end{abstract} 
 
  \section{Introduction}
  Generalized Petersen graphs were first defined by Watkins\cite{watkins1969theorem} who was
interested in trivalent graphs without proper three edge-colorings.
For integers $n$ and $k$ with $1 \leq k < n/2$, the \textit{generalized Petersen graph}
$GP(n,k)$ has been defined as an undirected graph with vertex-set $V=\{u_0,u_1,\ldots,u_{n-1},v_0,v_1,\ldots,v_{n-1}\}$ and edge set $E$ consisting of all pairs of the three forms $(u_i,u_{i+1}),(u_i,v_i)$ and $(v_i,v_{i+k})$ where $i$ is an integer and all subscripts are read modulo $n$. The above three forms of edges are called \textit{outer edges, spokes, and inner edges} respectively. In this original definition, $GP(n,k)$ is a trivalent graph of order $2n$ and size $3n$. It can be seen that when $n$ is even and $k = n/2$ the resulting graph is not cubic. And because of the obvious isomorphism  $GP(n,k)\cong GP(n,n-k)$, $k <n/2$.  In $GP(n, k)$, there exist one \textit{outer cycle} and  one or more \textit{inner cycles}. 
In this paper we may refer even (odd) subscripted vertices by even (odd) vertices.\begin{theorem}\cite{frucht1971groups}
If $d$ denotes the greatest common divisor of $n$ and $k$, then the set of inner edges generates a subgraph which
is the union of $d$ pairwise-disjoint $(n/d$)-cycles. 
\end{theorem}
A graph is said to be \textit{vertex-transitive} if its automorphism group acts transitively on the vertex set.
\begin{theorem}\cite{frucht1971groups}
$G(n, k)$ is vertex-transitive if and only if $k^2\equiv \pm 1 (\pmod n)$ or $n = 10$ and $k = 2$
\end{theorem}
 The well known \textit{Petersen graph} 
   $GP(5,2)$  is the smallest vertex-transitive graph which is not a Cayley graph\cite{saravzin1997note}. It has many interesting properties and can be taken as counter examples for many conjectures\cite{cruz2014examples, holton1993petersen}. 
The generalized Petersen graphs $GP(n, 1)$ are prisms,  isomorphic to the Cartesian product $C_ n\square K_2$. It can be easily seen that for even values of $n$ i.e, $n=2k$, $GP(2k,2)$ is planar for each $k$. Again Robertson\cite{robertson1968graphs} has shown that $GP(n, 2)$ is Hamiltonian unless $n \equiv 5\pmod 6$.    A set of vertices of the form $\{v_i, u_i,u_{i+1},\ldots,u_{i+n},v_{i+n}\}$
in a generalized Petersen graph is called an \textit{n-wicket}(or simply a \textit{wicket}) \cite{stueckle1984generalized}. 

The number of shortest paths or \textit{geodesics} between two vertices $u$ and $v$ in a graph will be  denoted by $\sigma (u,v)$.\\
\section{Generalized Petersen graph $GP(n,2)$ }
 The graph $GP(n,2)$ is defined for $n\geq 5$. It contains either one or two inner cycles  according as $n$ is odd or even (See Fig.\ref{pet2}).  If $n$ is odd, the inner cycle contains even vertices followed by odd vertices and when $n$ is even there are two inner cycles - the cycle of even vertices and the cycle of odd vertices each having $n/2$ vertices. (See Fig\ref{mob2}). Two inner vertices $v_i$ and $v_j$ are consecutive if $|i-j|=1$ or $|i-j|= -1 \mod(n)$ and  adjacent if $|i-j|=2$ or $|i-j|=-2 \mod(n)$. If $n$ is odd, say $n=2k+1$, then for each vertex $u_i$  there are two eccentric vertices $u_{i\pm k}$ in the outer cycle as diametric vertices. The  vertices on either side of $u_i$ are identical with respect to the metric. When $k$ is increased by one, the eccentric pair advances one more distance away from $u_i$.  If $n$ is even, say $n=2k$, then these diametric vertices coincide to a single vertex. 

\begin{figure}[H]
\centering										\scalebox{.8} 
{
\begin{pspicture}(0,-3.049375)(7.248125,3.049375)
\usefont{T1}{ptm}{m}{n}
\rput(3.651875,2.8509374){\small $u_0$}
\usefont{T1}{ptm}{m}{n}
\rput(1.76,2.36){\small $u_1$}
\usefont{T1}{ptm}{m}{n}
\rput(0.79,0.815){\small $u_2$}
\usefont{T1}{ptm}{m}{n}
\rput(0.79,-1.059){\small $u_3$}
\usefont{T1}{ptm}{m}{n}
\rput(1.9,-2.489){\small $u_4$}
\usefont{T1}{ptm}{m}{n}
\rput(3.811875,-2.8290625){\small $u_5$}
\usefont{T1}{ptm}{m}{n}
\rput(5.76,-2.12){\small $u_6$}
\usefont{T1}{ptm}{m}{n}
\rput(6.48,-0.889){\small $u_7$}
\usefont{T1}{ptm}{m}{n}
\rput(6.44,0.950){\small $u_8$}
\usefont{T1}{ptm}{m}{n}
\rput(5.54,2.1509376){\small $u_9$}
\usefont{T1}{ptm}{m}{n}
\rput(3.9,1.7000002){\small $v_0$}
\usefont{T1}{ptm}{m}{n}
\rput(2.82,1.52){\small $v_1$}
\pscircle[linewidth=0.03,dimen=outer](3.621875,0.00937495){2.6}
\psdots[dotsize=0.15](3.661875,1.5693749)
\psdots[dotsize=0.15](2.701875,1.229375)
\psdots[dotsize=0.15](2.161875,0.42937496)
\psdots[dotsize=0.15](2.181875,-0.63062507)
\psdots[dotsize=0.15](2.801875,-1.390625)
\psdots[dotsize=0.15](4.741875,-1.330625)
\psdots[dotsize=0.15](3.701875,-1.6706251)
\psdots[dotsize=0.15](5.261875,-0.47062504)
\psdots[dotsize=0.15](5.221875,0.58937496)
\psdots[dotsize=0.15](4.581875,1.349375)
\psbezier[linewidth=0.03](2.041875,0.45700195)(2.401875,0.42937496)(3.921875,-0.37062505)(3.6589541,1.5893749)
\psbezier[linewidth=0.03](5.241875,-0.45943964)(4.541875,-0.21062505)(3.5782237,0.20411295)(3.7582684,-1.710625)
\psbezier[linewidth=0.03](2.801875,-1.330625)(3.181875,-0.73062503)(4.081875,-0.05062505)(2.2085114,0.40814745)
\psbezier[linewidth=0.03](5.261875,0.60937494)(4.889655,0.30829436)(3.581875,0.30250755)(4.61708,-1.1906251)
\psbezier[linewidth=0.03](4.721875,-1.3091793)(4.201875,-0.8106251)(3.9499161,0.36937496)(2.801875,-1.350625)
\psbezier[linewidth=0.03](2.581875,1.1863309)(3.0362532,1.3087384)(3.521875,-0.43062505)(4.561875,1.349375)
\psbezier[linewidth=0.03](2.2578497,-0.6859081)(2.508365,-0.33255306)(4.001875,-0.13062505)(2.6819868,1.269375)
\psbezier[linewidth=0.03](3.6319807,1.5893749)(3.901875,1.109375)(3.301875,-0.27062505)(5.221875,0.53600097)
\psbezier[linewidth=0.03](4.547477,1.3389089)(4.421875,0.96937495)(3.5778854,0.02711255)(5.281875,-0.49062505)
\psbezier[linewidth=0.03](3.7972984,-1.7752975)(3.5217564,-1.3862299)(4.241875,0.20937495)(2.2114847,-0.62899673)
\psdots[dotsize=0.15](3.701875,2.569375)
\psdots[dotsize=0.15](2.041875,2.049375)
\psline[linewidth=0.03cm](2.021875,2.069375)(2.721875,1.229375)
\psline[linewidth=0.03cm](3.701875,2.609375)(3.661875,1.629375)
\psdots[dotsize=0.15](1.141875,0.70937496)
\psline[linewidth=0.03cm](1.121875,0.72937495)(2.201875,0.42937496)
\psdots[dotsize=0.15](1.241875,-1.010625)
\psline[linewidth=0.03cm](1.201875,-0.990625)(2.221875,-0.61062497)
\psdots[dotsize=0.15](2.341875,-2.2306252)
\psline[linewidth=0.03cm](2.821875,-1.350625)(2.321875,-2.2306252)
\psdots[dotsize=0.15](3.841875,-2.570625)
\psline[linewidth=0.03cm](3.721875,-1.5306251)(3.821875,-2.530625)
\psdots[dotsize=0.15](5.301875,-1.970625)
\psline[linewidth=0.03cm](4.601875,-1.1706251)(5.301875,-1.970625)
\psdots[dotsize=0.15](6.101875,-0.6906251)
\psline[linewidth=0.03cm](5.181875,-0.45062506)(6.061875,-0.67062503)
\psdots[dotsize=0.15](6.061875,0.92937493)
\psline[linewidth=0.03cm](5.161875,0.54937494)(6.141875,1.009375)
\psdots[dotsize=0.15](5.061875,2.129375)
\psline[linewidth=0.03cm](4.561875,1.3293749)(5.061875,2.129375)
\usefont{T1}{ptm}{m}{n}
\rput(5.22,-0.71999997){ $v_7$}
\usefont{T1}{ptm}{m}{n}
\rput(3.36,-1.7400001){ $v_5$}
\usefont{T1}{ptm}{m}{n}
\rput(2.12,-0.36000004){ $v_3$}
\usefont{T1}{ptm}{m}{n}
\rput(2.18,0.73999995){ $v_2$}
\usefont{T1}{ptm}{m}{n}
\rput(2.44,-1.2199999){ $v_4$}
\usefont{T1}{ptm}{m}{n}
\rput(4.46,-1.58){ $v_6$}
\usefont{T1}{ptm}{m}{n}
\rput(5.32,0.25999996){ $v_8$}
\usefont{T1}{ptm}{m}{n}
\rput(4.84,1.14){ $v_9$}
\end{pspicture} 
}
\centering
\scalebox{.8} 
{
\begin{pspicture}(0,-3.0448437)(7.4,3.0448437)
\pscircle[linewidth=0.03,dimen=outer](3.47375,0.05640635){2.55}
\psdots[dotsize=0.15](3.34375,2.5864062)
\psdots[dotsize=0.15](1.88375,2.0464063)
\psdots[dotsize=0.15](1.16375,1.1064066)
\psdots[dotsize=0.15](0.96375,-0.21359365)
\psdots[dotsize=0.15](1.40375,-1.3735937)
\psdots[dotsize=0.15](2.54375,-2.2735937)
\psdots[dotsize=0.15](5.26375,-1.7735934)
\psdots[dotsize=0.15](5.94375,-0.41359365)
\psdots[dotsize=0.15](5.76375,1.1264063)
\psdots[dotsize=0.15](4.92375,2.1664064)
\psdots[dotsize=0.15](3.44375,1.4264064)
\psdots[dotsize=0.15](4.24375,1.2264066)
\psdots[dotsize=0.15](4.74375,0.6864064)
\psdots[dotsize=0.15](4.90375,-0.15359364)
\psdots[dotsize=0.15](4.52375,-0.9535937)
\psdots[dotsize=0.15](3.68375,-1.3735937)
\psdots[dotsize=0.15](2.90375,-1.2335937)
\psdots[dotsize=0.15](2.30375,-0.71359366)
\psdots[dotsize=0.15](2.08375,0.00640635)
\psdots[dotsize=0.15](2.24375,0.66640633)
\psdots[dotsize=0.15](2.66375,1.1464063)
\psline[linewidth=0.03cm](3.34375,2.5864062)(3.44375,1.4064064)
\psline[linewidth=0.03cm](4.92375,2.1864064)(4.24375,1.2264066)
\psline[linewidth=0.03cm](5.76375,1.1264063)(4.72375,0.6864064)
\psline[linewidth=0.03cm](5.94375,-0.41359365)(4.92375,-0.17359366)
\psline[linewidth=0.03cm](5.26375,-1.7735934)(4.50375,-0.93359363)
\psline[linewidth=0.03cm](3.981456,-2.4372702)(3.706044,-1.3899173)
\psline[linewidth=0.03cm](2.5291781,-2.3515737)(2.9183218,-1.2356137)
\psline[linewidth=0.03cm](1.40375,-1.3535937)(2.30375,-0.6935936)
\psline[linewidth=0.03cm](0.94375,-0.21359365)(2.08375,0.02640635)
\psline[linewidth=0.03cm](1.12375,1.1464063)(2.26375,0.66640633)
\psline[linewidth=0.03cm](1.88375,2.0664062)(2.60375,1.1864064)
\psbezier[linewidth=0.03](3.44375,1.3864063)(3.44375,0.58640635)(3.88375,0.26640636)(4.74375,0.6864064)
\psbezier[linewidth=0.03](2.64375,1.1295791)(3.02375,0.82640636)(3.514636,0.1657046)(4.24375,1.2464063)
\psbezier[linewidth=0.03](2.18375,0.6816686)(3.22075,0.16640635)(3.42375,0.6625726)(3.42375,1.4264064)
\psbezier[linewidth=0.03](2.08375,-0.01359365)(3.28375,0.18640634)(3.08375,0.46640635)(2.66375,1.1264063)
\psbezier[linewidth=0.03](2.20375,0.66640633)(3.68375,0.04640635)(2.5553918,-0.4604508)(2.2962873,-0.6935936)
\psbezier[linewidth=0.03](2.14375,-0.026758205)(3.88375,0.30640635)(2.8896356,-1.3735937)(2.8066363,-1.1675828)
\psbezier[linewidth=0.03](2.32375,-0.6087242)(2.3480413,-0.76324105)(3.4928927,0.58640635)(3.64375,-1.3535937)
\psbezier[linewidth=0.03](2.88375,-1.2535937)(3.412,0.36640635)(4.48375,-0.93359363)(4.4657726,-0.9535937)
\psbezier[linewidth=0.03](3.659431,-1.3135936)(3.6425643,-1.1655158)(3.26375,0.02640635)(4.84375,-0.16262415)
\psbezier[linewidth=0.03](4.458988,-0.9535937)(3.972477,-0.4138468)(3.68375,0.28640634)(4.76375,0.6864064)
\psbezier[linewidth=0.03](4.96375,-0.21359365)(3.26375,0.04640635)(4.28375,1.1664064)(4.2560396,1.2464063)
\usefont{T1}{ptm}{m}{n}
\rput(3.22375,2.8464062){ $u_0$}
\usefont{T1}{ptm}{m}{n}
\rput(1.571875,2.2864063){ $u_1$}
\usefont{T1}{ptm}{m}{n}
\rput(0.70375,1.2864064){ $u_2$}
\usefont{T1}{ptm}{m}{n}
\rput(0.471875,-0.32453114){ $u_3$}
\usefont{T1}{ptm}{m}{n}
\rput(0.871875,-1.4735936){ $u_4$}
\usefont{T1}{ptm}{m}{n}
\rput(2.2375,-2.5135937){ $u_5$}
\usefont{T1}{ptm}{m}{n}
\rput(4.151875,-2.75){ $u_{-5}$}
\usefont{T1}{ptm}{m}{n}
\rput(5.6,-2.10){ $u_{-4}$}
\usefont{T1}{ptm}{m}{n}
\rput(6.4,-0.53359365){ $u_{-3}$}
\usefont{T1}{ptm}{m}{n}
\rput(6.2,1.1864064){ $u_{-2}$}
\usefont{T1}{ptm}{m}{n}
\rput(5.021875,2.4754689){ $u_{-1}$}
\usefont{T1}{ptm}{m}{n}
\rput(3.671875,1.5754689){ $v_0$}
\usefont{T1}{ptm}{m}{n}
\rput(2.671875,-1.0245311){ $v_5$}
\usefont{T1}{ptm}{m}{n}
\rput(1.971875,0.93546885){ $v_2$}
\usefont{T1}{ptm}{m}{n}
\rput(1.871875,0.24640635){ $v_3$}
\usefont{T1}{ptm}{m}{n}
\rput(2.15375,-0.45359364){ $v_4$}
\usefont{T1}{ptm}{m}{n}
\rput(4.134,-1.3245312){ $v_{-5}$}
\usefont{T1}{ptm}{m}{n}
\rput(2.571875,1.465469){ $v_1$}
\usefont{T1}{ptm}{m}{n}
\rput(4.9775,-1.0245311){ $v_{-4}$}
\usefont{T1}{ptm}{m}{n}
\rput(5.07375,-0.42){ $v_{-3}$}
\usefont{T1}{ptm}{m}{n}
\rput(4.97375,0.40640635){ $v_{-2}$}
\usefont{T1}{ptm}{m}{n}
\rput(4.521875,0.95){ $v_{-1}$}
\psdots[dotsize=0.15](3.96375,-2.4335938)
\end{pspicture} 
}

\caption{Labelling : $GP(10,2)$ and $GP(11,2)$}\label{pet2}
\end{figure}
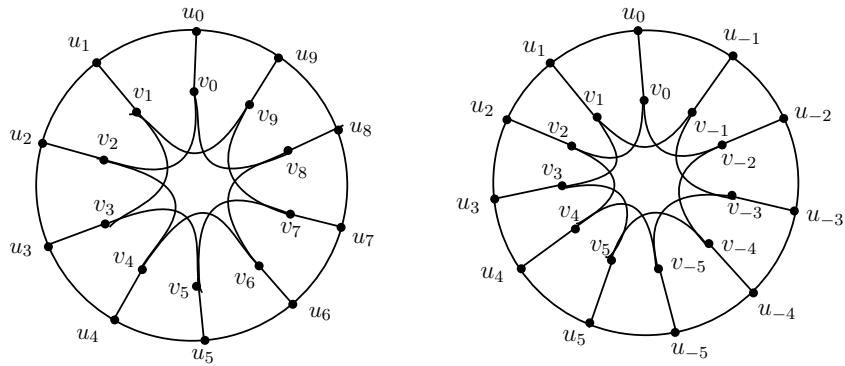
$GP(n,2)$ when $n$ is odd  can be viewed as a M\"{o}bius strip with outer vertices lying in the middle and inner vertices lying on the border.
If $n=2k+1$, it is easy to see that there are $k$ odd vertices lying on the upper ends and $k+1$ even vertices lying on the lower ends of the  strip (See Fig\ref{mob1}). Moving along the middle of the M\"{o}bius strip, the outer vertices comes in a regular manner as $u_0,u_1,u_2,\ldots,u_{2k}$ and along the border of the strip,  odd vertices follow the even vertices. 
 One movement along the edge in the border of the strip is equivalent to two movements along the edges in the middle. Hence shortest paths prefer edges along the border when the number of vertices increases.\\

\begin{figure}[H]
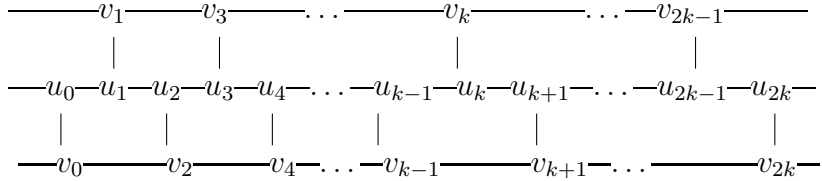


\rule[0.5ex]{1.2cm}{0.4pt}$v_1$\rule[-.3\baselineskip]{0pt}{\baselineskip}\rule[0.5ex]{1cm}{0.4pt}$v_3$\rule[0.5ex]{1cm}{0.4pt}\ldots\rule[0.5ex]{1.3cm}{0.4pt}$v_{k}$\rule[0.5ex]{1.5cm}{0.4pt}\ldots\rule[0.5ex]{0.4cm}{0.4pt}$v_{2k-1}$\rule[0.5ex]{1.1cm}{0.4pt}\\\\
\put(40,10){\line(0,1){12}}\put(80,10){\line(0,1){12}}\put(170,10){\line(0,1){12}}\put(260,10){\line(0,1){12}}
\rule[0.5ex]{0.5cm}{0.4pt}$u_0$\rule[0.5ex]{0.3cm}{0.4pt}$u_1$\rule[0.5ex]{0.3cm}{0.4pt}$u_2$\rule[0.5ex]{0.3cm}{0.4pt}$u_3$\rule[0.5ex]{0.3cm}{0.4pt}$u_4$\rule[0.5ex]{0.3cm}{0.4pt}\ldots\rule[0.5ex]{0.3cm}{0.4pt}$u_{k-1}$\rule[0.5ex]{0.3cm}{0.4pt}$u_{k}$\rule[0.5ex]{0.3cm}{0.4pt}$u_{k+1}$\rule[0.5ex]{0.3cm}{0.4pt}\ldots\rule[0.5ex]{0.3cm}{0.4pt}$u_{2k-1}$\rule[0.5ex]{0.3cm}{0.4pt}$u_{2k}$\rule[0.5ex]{0.3cm}{0.4pt}\\\\
\put(20,10){\line(0,1){12}}\put(60,10){\line(0,1){12}}\put(100,10){\line(0,1){12}}\put(140,10){\line(0,1){12}}\put(200,10){\line(0,1){12}}\put(290,10){\line(0,1){12}}
\rule[-.3\baselineskip]{0pt}{\baselineskip}
\rule[0.5ex]{0.5cm}{0.4pt}$v_0$\rule[0.5ex]{1.1cm}{0.4pt}$v_2$\rule[0.5ex]{1cm}{0.4pt}$v_4$\rule[0.5ex]{0.3cm}{0.4pt}\ldots\rule[0.5ex]{0.3cm}{0.4pt}$v_{k-1}$\rule[0.5ex]{1.2cm}{0.4pt}$v_{k+1}$\rule[0.5ex]{0.3cm}{0.4pt}\ldots\rule[0.5ex]{1.4cm}{0.4pt}$v_{2k}$\rule[0.5ex]{0.3cm}{0.4pt}

\caption{Strip for $GP(2k+1,2)$, when $k$ is odd}\label{mob1}
\end{figure}
\begin{figure}[h]

\centering
\scalebox{.8}
{
\begin{pspicture}(0,-2.88)(6.08,2.88)
\definecolor{color2}{rgb}{0.09411764705882353,0.0392156862745098,0.0392156862745098}
\pscircle[linewidth=0.03,dimen=outer](3.061875,0.0){2.02}
\pscircle[linewidth=0.03,linecolor=color2,dimen=outer](3.071875,0.05){1.01}
\pscircle[linewidth=0.03,linecolor=color2,dimen=outer](3.021875,0.0){2.88}
\psdots[dotsize=0.15,linecolor=color2](3.081875,2.02)
\psdots[dotsize=0.15,linecolor=color2](3.061875,1.02)
\psdots[dotsize=0.15,linecolor=color2](1.601875,1.36)
\psdots[dotsize=0.15,linecolor=color2](1.061875,0.1)
\psdots[dotsize=0.15,linecolor=color2](0.921875,1.96)
\psline[linewidth=0.022cm,linecolor=color2](3.061875,2.06)(3.041875,1.04)
\psline[linewidth=0.022cm,linecolor=color2](0.901875,1.98)(1.561875,1.38)
\psdots[dotsize=0.15,linecolor=color2](2.081875,0.12)
\psdots[dotsize=0.15,linecolor=color2](3.061875,-0.94)
\psdots[dotsize=0.15,linecolor=color2](4.041875,0.14)
\psdots[dotsize=0.15,linecolor=color2](5.061875,0.18)
\psdots[dotsize=0.15,linecolor=color2](3.081875,-2.02)
\psline[linewidth=0.022cm,linecolor=color2](1.041875,0.1)(2.041875,0.12)
\psline[linewidth=0.022cm,linecolor=color2](4.081875,0.12)(5.081875,0.16)
\psline[linewidth=0.022cm,linecolor=color2](3.061875,-0.94)(3.081875,-2.1)
\psdots[dotsize=0.15,linecolor=color2](1.601875,-1.38)
\psdots[dotsize=0.15,linecolor=color2](4.501875,-1.42)
\psdots[dotsize=0.15,linecolor=color2](4.461875,1.46)
\psdots[dotsize=0.15,linecolor=color2](5.141875,1.96)
\psdots[dotsize=0.15,linecolor=color2](5.121875,-1.98)
\psdots[dotsize=0.15,linecolor=color2](0.941875,-1.94)
\psline[linewidth=0.022cm,linecolor=color2](4.441875,1.46)(5.121875,1.94)
\psline[linewidth=0.022cm,linecolor=color2](4.481875,-1.38)(5.061875,-1.96)
\psline[linewidth=0.022cm,linecolor=color2](0.921875,-1.9)(1.621875,-1.34)
\usefont{T1}{ptm}{m}{n}
\rput(4.89,-1.34){ $u_5$}
\usefont{T1}{ptm}{m}{n}
\rput(1.721875,-0.12){ $v_2$}
\usefont{T1}{ptm}{m}{n}
\rput(1.091875,-1.36){ $u_3$}
\usefont{T1}{ptm}{m}{n}
\rput(3.521875,1.12){ $v_0$}
\usefont{T1}{ptm}{m}{n}
\rput(4.381875,-0.12){ $v_6$}
\usefont{T1}{ptm}{m}{n}
\rput(4.931875,1.4){ $u_7$}
\usefont{T1}{ptm}{m}{n}
\rput(5.521875,2.0){ $v_7$}
\usefont{T1}{ptm}{m}{n}
\rput(5.351875,0.16){ $u_6$}
\usefont{T1}{ptm}{m}{n}
\rput(5.321875,-2.16){ $v_5$}
\usefont{T1}{ptm}{m}{n}
\rput(3.071875,-2.28){ $u_4$}
\usefont{T1}{ptm}{m}{n}
\rput(0.461875,-2.08){ $v_3$}
\usefont{T1}{ptm}{m}{n}
\rput(0.61875,0.06){ $u_2$}
\usefont{T1}{ptm}{m}{n}
\rput(0.601875,2.2){ $v_1$}
\usefont{T1}{ptm}{m}{n}
\rput(3.071875,2.24){ $u_0$}
\usefont{T1}{ptm}{m}{n}
\rput(1.171875,1.32){ $u_1$}
\usefont{T1}{ptm}{m}{n}
\rput(3.421875,-1.19){ $v_4$}
\end{pspicture} 
}
\caption{A representation of $GP(2k,2)$ when $k=4$}\label{mob2}
\end{figure}
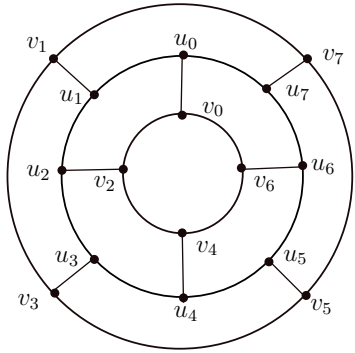
\section{Geodesics in $GP(n,2)$}
\subsection{Geodesics in $GP(n,2)$ when $n$ is odd}
\subsubsection{Number of geodesics between a pair of vertices in the outer cycle of $GP(n,2)$ when $n$ is odd}
In $GP(n,2)$ where $n=2k+1$, by symmetry we consider the vertices $u_0$ and $u_r$ where $1\leq r \leq k,\;k\geq 2 $ and any shortest path joining them may be denoted by $P(u_0,u_r)$.
\begin{lemma}
In $GP(2k+1,2),\,k\geq 2$  for the vertices $u_0$ and $u_r$  in the outer cycle,  $1\leq r \leq k $,  there is always a geodesic joining them contained in the outer cycle for  $r\leq 5$. 
\end{lemma}\label{l1}
\begin{proof}
By symmetry, we consider $r$ for $1\leq r \leq k$. 
It is obvious for $r=1$ since $(u_0,u_1)$ makes an outer edge. 
For $r=2$, the only geodesic is the one joining $\{u_0,u_1,u_2\}$ since any path intersecting the inner cycle contains two  spokes and atleast one inner edge. For $r=3$, the only geodesic is the one joining $\{u_0,u_1,u_2,u_3\}$ since the inner vertices $v_0$ and $v_3$ are non-adjacent, they make a path of minimum length 4 joining either $\{u_0,u_1,v_1,v_3,u_3\}$, $\{u_0,v_0,v_2,u_2,u_3\}$ or $\{u_0,v_0,v_5,v_3,u_3\}$ in the case $k=3$. For $r=4$ there are two geodesics of length 4, one joining $\{u_0,u_1,u_2,u_3,u_4\}$ and the other joining $\{u_0,v_0,v_2,v_4,u_4\}$. For $r=5(<k)$, there are three geodesics joining $\{u_0,u_1,u_2,u_3,u_4,u_5\}$, $\{u_0,u_1,v_1,v_3,v_5,u_5\}$ and $\{u_0,v_0,v_2,v_4,u_4,u_5\}$ each of length 5. When $r=5(=k)$, $u_0$ and $u_5$ become the extreme vertices and hence there are four geodesics  between  $u_0$ and $u_5$ including the one passing through the inner cycle in the reverse direction. 
\end{proof}
\begin{lemma}
In $GP(2k+1,2)$, for even $r\leq k$,  there is a unique geodesic joining  $u_0$ and $u_r$ for $r>5$  and it passes through the  spokes at $u_0$ and $u_r$ and the inner vertices lying between $v_0$ and $v_r$.
\end{lemma}
\begin{proof}
Consider $u_r$ where $r$ is even and $r>5$. Let $P(u_0,u_r)$ be any geodesic joining $u_0$ and $u_r$  contained in the outer cycle.  If $P(u_0,u_r)$ is contained in the outer cycle, its length becomes $r$. Since $r$ is even, $v_0$ and $v_r$ are even vertices and they lie on  a unique shortest path of length $r/2$ contained in the inner cycle. Considering the spokes at $v_0$ and $v_r$, the length of the path $P(u_0,u_r)$ becomes $r/2+2<r$ for $r>5$. Thus $P(u_0,u_r)$ passes through the inner vertices and the spokes at $u_0$ and $u_r$ for $r>5$. 
\end{proof}
\begin{lemma}
In $GP(2k+1,2)$, for odd $r$,  there are two geodesics between $u_0$ and $u_r$ for $5<r<k$ and three for $r=k$ all  passing through the inner vertices.
\end{lemma}
\begin{proof}
When $r$ is odd and $5<r<k$, it can be easily seen that $P(u_0,u_r)$ contains exactly two spokes either at $u_0$ and $u_{r-1}$ or at $u_1$ and $u_{r}$ having length $\frac{r-1}{2}+3$. Otherwise $r<\frac{r-1}{2}+3$, a contradiction. When $r$ is odd, the distance between $v_0$ and $v_r$ in the inner cycle is $k-\frac{r-1}{2}$.  When $r=k$, including the spokes at $u_0$ and $u_r$, the distance becomes
$\frac{r-1}{2}+3$. Thus there are three geodesics between $u_0$ and $u_r$ when $r=k$.
\end{proof}
\begin{proposition}
In $GP(2k+1,2)$, $k\geq 2$, for the vertices $u_0$ and $u_r$  in the outer cycle,  $1\leq r \leq k$,  there is no geodesic joining them contained in the outer cycle for $r> 5$. 
\end{proposition}
\begin{theorem}
If $u_i$ and $u_j$ are any two distinct vertices in the outer cycle of $GP(2k+1,2)$ where $|i-j|=r\leq k$, then the number of geodesics $\sigma(u_i,u_j)$ between $u_i$ and $u_j$  is given by\\
\begin{equation*}
 \sigma(u_i,u_j)=\begin{cases}
   1&\hbox{ for }r=1,2,3 \\
   2&\hbox{ for }r=4\\
   3&\hbox{ for }r=5;\;r<k\\
  4&\hbox{ for }r=5;\;r=k\\
  1&\hbox{ for }r=6,8,10,\ldots\\
  2&\hbox{ for }r=7,9,11,\ldots;\;r<k\\
  3&\hbox{ for }r=7,9,11,\ldots;\;r=k
  \end{cases}
\end{equation*} 
\end{theorem}
\begin{proof}
By symmetry, we consider $r\leq k$. When $n$ is odd, there exists only one inner cycle in $GP(n,2)$. There is a unique geodesic between $u_i$ and $u_{i+r}$ for $r=1,2,3$ lying on the outer cycle (Lemma \ref{l1}) and two geodesics joining $u_i$ and $u_{i+4}$ namely, \begin{center}
$\{u_i,u_{i+1},u_{i+2},u_{i+3},u_{i+4}\}$ and $\{u_i,v_i,v_{i+2},v_{i+4},u_{i+4}\}$
\end{center} 
When $r=5(=k)$, the vertices $u_i$ and $ u_{i+5}$ becomes diametric pair on the outer cycle and hence there are four geodesics  joining them namely, $\{u_i,u_{i+1},u_{i+2},u_{i+3},u_{i+4},u_{i+5}\}$  lying on the outer cycle and $\{u_i,u_{i+1},v_{i+1},v_{i+3},v_{i+5},u_{i+5}\}$, $\{u_i,v_i,v_{i+2}$, $v_{i+4},u_{i+4},u_{i+5}\}$, through inner cycle in the forward direction and $\{u_i,v_i,v_{i-2},v_{i-4},v_{i-6},u_{i-6}(=u_{i+5})\}$ in the reverse direction. 

But when $r<k$, the path  in the reverse direction does not become a geodesic. Therefore there are only three geodesics between $u_i$ and $u_{i+5}$ when $r<k$. 

When $r>5$, all geodesics pass through the inner cycle and no geodesic lies entirely on the outer cycle. 

When $r(>5)$ is even, $u_i$ and $u_{i+r}$ have the same parity and hence there exists only one geodesic joining $u_i$ and $u_{i+r}$ namely, $$\{u_i,v_{i},v_{i+2},v_{i+4},\ldots,v_{i+r},u_{i+r}\}$$

When $r$ is odd and  $5<r< k$, there are two geodesics namely, \begin{center}
$\{u_i,u_{i+1},v_{i+1},v_{i+3},\ldots,v_{i+r},u_{i+r}\}$ and  $\{u_i,v_{i},v_{i+2},v_{i+4},\ldots,v_{i+r-1},u_{i+r-1},u_{i+r}\}$.
\end{center} 

When $r=k$, there is one more geodesics in the reverse direction i.e, $$\{u_i,v_i,v_{i-2},v_{i-4},\ldots,v_{i-r-1},u_{i-r-1}(=u_{i+r})\}$$  
\end{proof}

\subsubsection{Number of geodesics between a pair of vertices in the inner cycle of $GP(n,2)$ when $n$ is odd}
\begin{theorem}
If $v_i$ and $v_j$ are any two distinct vertices in the inner cycle of $GP(2k+1,2)$ where  $|i-j|=r\leq k$,
then number of geodesics $\sigma(v_i,v_j)$ between $v_i$ and $v_j$ is given by\\
\begin{equation*}
\sigma(v_i,v_j)=\begin{cases}
      
   1 &\hbox{for even } r\\
   1 &\hbox{for odd } r \hbox{, } r>k-2\\
  \frac{r+1}{2} &\hbox{for odd } r \hbox{, } r<k-2\\
  \frac{r+3}{2} &\hbox{for odd } r \hbox{, }r=k-2
\end{cases}
\end{equation*}
\end{theorem}
\begin{proof}
When $n$ is odd, there exists only one inner cycle in $GP(n,2)$ containing all even vertices followed by all odd vertices.\\

 When $n=2k+1$, there are $k+1$ even vertices and $k$ odd vertices in the inner cycle. Consider two consecutive inner vertices $v_i$ and $v_{i+1}$.  Since $v_i$ and $v_{i+1}$ are non adjacent inner vertices, there is a unique geodesic between $v_i$ and $v_{i+1}$ passing through the 1-wicket $\{v_i,u_i,u_{i+1},v_{i+1}\}$  containing two spokes and an outer edge.\\
 
  When $r$ is even, both $v_i$ and $v_{i+r}$ have the same parity and therefore  there is a unique geodesic $P(v_i,v_{i+r})$ of length $r/2$ namely $\{v_i,v_{i+2},v_{i+4},\ldots,v_{i+r}\}$.\\
  
   When $r$ is odd, $r<k-2$; $v_i$ and $v_{i+r}$ have opposite parity and a geodesic from $v_i$ to $v_{i+r}$ passes through the 1-wicket at any one of the consecutive pairs\\ $(v_i, v_{i+1}),(v_{i+2}, v_{i+3}),\ldots, (v_{i+r-1},v_{i+r})$ 
and thus there exist $\frac{r+1}{2}$ geodesics joining $v_i$ and $v_{i+r}$ having length $ \frac{r+1}{2}+3$. When $r=k-2$, the pair $(v_i,v_{i+r})$  lies  sufficiently apart so that there is one more geodesic in the reverse direction. When $r>k-2$, the geodesic in the reverse direction alone exists.
\end{proof}

\subsubsection{Number of geodesics between a pair of vertices in the outer and inner cycle of $GP(n,2)$ when $n$ is odd}
\begin{theorem}
If $u_i$ and $v_j$ are any two  vertices in the outer and inner cycles respectively of $GP(2k+1,2)$ where  $|i-j|=r\leq k$,
then the number of geodesics  $\sigma(u_i,v_j)$ between $u_i$ and $v_j$ is given by\\
\begin{equation*}
\sigma(u_i,v_j)=\begin{cases}     
 1&\hbox{ for} \; r<k\\
 1&\hbox{for even } r\hbox{, } \; r=k\\
 2&\hbox{for odd } r\hbox{,} \; r=k
\end{cases}
\end{equation*}
\end{theorem}
\begin{proof}
When $n$ is odd, there exists only one inner cycle in $GP(n,2)$ containing all even vertices followed by all  odd vertices.  Consider $u_i$ and $v_{i+r}$. If both vertices are either even or odd, $v_i$ and $v_{i+r}$ are also the same and there is a unique geodesic of length $r/2+1$ passing through the spoke $(u_i,v_i)$ joining $v_i$ and $v_{i+r}$ along the inner cycle . Otherwise the geodesic passes through the outer edge $(u_i,u_{i+1})$,  the spoke $(u_{i+1},v_{i+1})$ and joins $v_{i+1}$ to $v_{i+r}$ along the inner cycle. It is of length $\frac{r-1}{2}+2$. In the extreme case $r=k$, there is one more geodesic in the reverse direction passing through the spoke $(u_i,v_i)$.
\end{proof}
\subsection{Geodesics in $GP(n,2)$ when $n$ is even}
\subsubsection{Number of geodesics between a pair of vertices in the outer cycle of $GP(n,2)$ when $n$ is even}
\begin{theorem}
If $u_i$ and $u_j$ are any two distinct vertices in the outer cycle of $GP(2k,2)$,  $k\geq 3$ where $|i-j|=r\leq k$, then the number of geodesics $\sigma(u_i,u_j)$ between $u_i$ and $u_j$ is given by\\
\begin{equation*}
\sigma(u_i,u_{i+r})=\begin{cases}   
   1&\hbox{ for } r=1,2,3;\;r<k \\
   2&\hbox{ for } r=3;\;r=k\\
   2&\hbox{ for } r=4;\;r<k\\
   4&\hbox{ for } r=4;\;r=k\\
   3&\hbox{ for } r=5;\;  r<k\\
   6&\hbox{ for } r=5;\;r=k\\
   1&\hbox{ for } r=6,8,10,\ldots;\;r<k\\
   2&\hbox{ for }  r=6,8,10,\ldots;\;r=k\\
   2&\hbox{ for }  r=7,9,11,\ldots;\;r<k\\
   4&\hbox{ for }  r=7,9,11,\ldots;\;r=k
\end{cases}
\end{equation*}
\end{theorem}
\begin{proof}
 Since $n$ is even, there are two inner cycles - the cycle of even vertices and the cycle of odd vertices. When $r=1,2,3;\;r<k$, there is a unique geodesic between $u_i$ and $u_{i+r}$ lying on the outer cycle and in the extreme case, i.e, when $r=3$, $r=k$, the outer cycle itself makes two geodesics on either sides.\\
 
  When $r=4$, $r<k$, there are two geodesics of length $4$, one over the outer cycle and the other over the inner cycle of even or odd vertices according as $i$ is even or odd by means of  the two spokes at the given vertices. In its extreme case there are two more geodesics passing over the outer cycle and the inner cycle in the reverse direction.\\
  
   When $r=5,\;r<k$, there are $3$ geodesics of length $5$. One lying on the outer cycle and the others passing  through the spokes either at $u_{i+1}$ and $u_{i+5}$ or at $u_i$ and $u_{i+4}$. In its extreme case when $r=k$, there are $3$ more geodesics passing over the outer and inner cycles in the reverse direction.\\
 
  When $r=6,8,\ldots;\;r<k$, there is no geodesic passing over the outer cycle. Since $v_{i}$ and $v_{i+r}$ lie on the same inner cycle, there is a geodesic joining them of length $r/2$. Therefore there is a unique geodesic of length $r/2+2$ joining $u_i$ and $u_{i+r}$ when $r=6,8,\ldots;\;r<k$. In the extreme case there is one more geodesic passing over the inner cycle in the reverse direction.\\
 
  When $r=7,9,\ldots;\;r<k$, $v_i$ and $v_{i+r}$ lie on different inner cycles and therefore there are two geodesics one joining the vertices $\{u_i,u_{i+1},v_{i+1},v_{i+3},\ldots,v_{i+r},u_{i+r}\}$ and the other joining the vertices $\{u_i,v_{i},v_{i+2},v_{i+4},\ldots,v_{i+r-1},u_{i+r-1},u_{i+r}\}$. In the extreme case, reversing the direction over the outer and inner cycles, there lie two more geodesics. 
\end{proof}
\subsubsection{Number of geodesics between a pair of vertices in the inner cycle of $GP(n,2)$ when $n$ is even}
\begin{theorem}
If $v_i$ and $v_j$ are any two distinct vertices in the inner cycle of $GP(2k,2)$,  $k\geq 3$ where $|i-j|=r\leq k$, then the number of geodesics $\sigma(v_i,v_j)$ between $v_i$ and $v_j$ is given by
\begin{equation*}
\sigma(v_i,v_j)=\begin{cases}
1& \hbox{for even }r, r<k\\
2& \hbox{for even }r, r=k\\
\frac{r+1}{2}& \hbox{for odd }r, r<k\\
r+1& \hbox{for odd }r, r=k
\end{cases}
\end{equation*}
\end{theorem}
\begin{proof}
When $r$ is even and $r<k$ both the vertices $v_i$ and $v_{i+r}$ lie on the same inner cycle. Therefore, there exists a unique geodesic lying on the same inner cycle. In the extreme case, there is one more geodesic on the reverse side of the inner cycle.\\

 When $r$ is odd and $r<k$ the vertices lie on different inner cycles and hence joined by a wicket containing two spokes and an outer edge at any vertex $v_i,v_{i+2},\ldots,v_{i+r-1}$. Hence there are $\frac{r+1}{2}$ geodesics. In the extreme case the above method can be repeated along the reverse side of the inner cycle.
\end{proof}
\subsubsection{Number of geodesics between a pair of vertices in the outer and inner cycle of $GP(n,2)$ when $n$ is even}
\begin{theorem}
If $u_i$ and $v_j$ are any two  vertices in the outer and inner cycles respectively of $GP(2k,2)$, $k\geq 3$ where $|i-j|=r\leq k$,
then the number of geodesics  $\sigma(u_i,v_j)$ between $u_i$ and $v_j$ is given by
\begin{equation*}
\sigma(u_i,v_j)=\begin{cases}
1&r<k\\2&r=k
\end{cases}
\end{equation*}
\end{theorem}
\begin{proof}
When $r$ is even and $r<k$, both $v_i$ and $v_{i+r}$ are either even or odd and hence there is a unique geodesic joining the spoke $(u_i,v_i)$ to $v_{i+r}$.  In the extreme case there is one more geodesic along the opposite side of the same inner cycle.\\

 When $r$ is odd and $r<k$, $u_i$ and $v_{i+r}$ lie on different inner cycles. So there is a unique geodesic passing through $\{u_i,u_{i+1},v_{i+1}\}$. In the extreme case there is one more geodesic lying in the opposite direction. 
\end{proof}
It can be seen that in $GP(n,2)$ when $n$ is even, the outer and inner cycles are even and hence the number of geodesics in each of the extreme cases doubles.

 \section{Distance between two vertices in  $GP(n,2)$}

\begin{theorem}

The distance between a pair of vertices $(u_i,u_j)$,  $(v_i,v_j)$ and $(u_i,v_j)$ in  $GP(n,2)$ is given by  

\begin{equation*}
d(u_i,u_j)=\begin{cases}
   
   r&\hbox{ for } r\leq 5 \\
\frac{r+4}{2} &\hbox{ for even } r \hbox{, } r>5\\
\frac{r+5}{2}&\hbox{ for odd } r \hbox{, }  r>5\\
\end{cases}
\end{equation*}
\begin{equation*}
d(v_i,v_j)=\begin{cases}    
  \frac{r}{2}&\hbox{ for even } r \\
\frac{r+5}{2}&\hbox{ for odd } r
\end{cases}
\end{equation*}
\begin{equation*}
d(u_i,v_j)=\begin{cases}     
   \frac{r+2}{2}&\hbox{ for even } r  \\
\frac{r+3}{2}&\hbox{ for odd } r \\
\end{cases}
\end{equation*}
 where $|i-j|=r\leq \ceil*{\frac{n-1}{2}}$
\end{theorem}
\begin{proof}
For $r\leq \ceil*{\frac{n-1}{2}}$, from any vertex $u_i$ to  $u_{i+r}$ (or $u_{i-r}$) for $r\leq 5$ there is a geodesic of length $r$ lying on the outer cycle.  When $r$ is even,  $r>5$ there is a geodesic of length $\frac{r}{2}+2$ joining $u_i$ and $u_{i+r}$ passing through $\{v_i,v_{i+2},\ldots,v_{i+r}\}$ and the two spokes at $u_i$ and $u_{i+r}$.\\
When $r$ is odd, there is a geodesic of length $\frac{r-1}{2}+3$ joining $u_i$ and $u_{i+r}$ passing through $\{u_{i+1},v_{i+1},v_{i+3},\ldots,v_{i+r-1},u_{i+r-1}\}$.\\

 If both $v_i$ and $v_{i+r}$ are either odd or even there is a unique geodesic of length $r/2$ joining them along the inner cycle, if not so, there is a geodesic $\{v_i,u_i,u_{i+1},v_{i+1},v_{i+3},\ldots,v_{i+r}\}$ of length $\frac{r-1}{2}+3$. \\
 
   If $r$ is even, $u_i$ and $v_i$ both are either odd or even and there is a geodesic of length $\frac{r}{2}+1$, otherwise there is a geodesic of length $\frac{r-1}{2}+2$ including a spoke and an outer edge.
\end{proof}
\begin{corollary}
The diameter of $GP(n,2)$ when $n\geq 8,$ is given by\\
$$diam\, GP(n,2)=\ceil*{\frac{n-1}{4}}+2$$
\end{corollary}
\section{Betweenness centrality}
Betweenness centrality\cite{freeman1978centrality,borgatti2006graph,hage1995eccentricity,friedkin1991theoretical} measures the relative importance of vertex in a graph.  A vertex is said to be  central  if it can effectively monitor the communication between  vertices. It describes how a vertex acts as a bridge among all the pairs of vertices.  Betweenness centrality of a vertex $x$ is the sum of the fraction of all-pairs shortest paths that pass through $x$. It has wide applications in the analysis of networks\cite{chen2012identifying,everton2009network,memon2008detecting,estrada2006virtual,khansari2016centrality}. \\\\
\textbf{Betweenness centrality of a vertex in a graph}
\begin{definition}\cite{freeman1977set}.
 Let $G$ be a graph and $x\in V(G)$, then the betweenness centrality of $x$ in $G$ denoted by $B_G(x)$ or simply $B(x)$  may be defined as
$$B_G(x) = \sum_{s,t \in V(G)\setminus\{x\}}{\frac{\sigma_{ st}(x)}{\sigma_{ st}}}$$ where  $\sigma_{ st}(x)$ denotes the number of shortest $s$-$t$ paths in $G$ passing through $x$ and $\sigma_{ st}$, the number of shortest $s$-$t$ paths in  $G$. The ratio $\frac{\sigma_{ st}(x)}{\sigma_{ st}}$ is called \textit{pair dependency} of $\{s,t\}$ on $x$, denoted by $\delta_G(s,t,x)$. 
\end{definition}
We may now define the following terms related to betweenness centrality.\\

\begin{definition}
Let $G$ be a graph and $H$ a subgraph of $G$. Let $x\in V(H)$, then the betweenness centrality of $x$ in $H$ denoted by $B_H(x)$  may be defined as\\ 
$$B_H(x) =\sum_{s,t\in V(H)\setminus\{x\}}{\frac{\sigma^H_{ st}(x)}{\sigma^H_{ st}}}$$ where  $\sigma^H_{ st}(x)$ and $\sigma^H_{ st}$ denotes  the number of shortest $s$-$t$ paths  passing through $x$ and  the number of shortest $s$-$t$ paths respectively, being their vertices in $H$.
\end{definition}
\begin{definition}
Let $G$ be a graph and $H$ a subgraph of $G$. Let $x\in V(G)$, then the betweenness centrality of $x$ induced by $H$ denoted by $B(x,H)$ may be defined as\\ 
$$B(x,H) =  \sum_{s,t(\neq{x})\in V(H)}{\frac{\sigma_{ st}(x)}{\sigma_{ st}}}$$
\end{definition}
\begin{definition}
Let $G$ be a graph and $S$ a subset of $V(G)$. Let $x\in V(G)$, then the betweenness centrality of $x$ induced by $S$ denoted by $B(x,S)$  may be  defined as\\ 
$$B(x,S) = \sum_{s,t(\neq x)\in S}{\frac{\sigma_{ st}(x)}{\sigma_{ st}}}$$
\end{definition}

\begin{definition}
Let $G$ be a graph and $x,x_0 \in V(G)$,  then the betweenness centrality of  $x$ induced by $x_0$ in $G$, denoted by $B_G(x,x_0)$ or simply $B(x,x_0)$ is defined by
$$B_G(x,x_0)=\sum_{t\in V(G)\setminus x}\frac{\sigma_{ x_0t}(x)}{\sigma_{ x_0t}}$$
\end{definition} 
It can be easily seen that in any graph $G$, the betweenness centrality induced by a vertex on its extreme vertex or an end vertex is zero. $B(x_i,x_j) = 0$ for complete graph $K_n$. 
 Let $P_n$ be a path on $n$ vertices $\{x_1,\ldots,x_n\}$, then 
 
\begin{equation*}
 B(x_i,x_j) =  \begin{cases}
       i-1 & \hbox{ if } i<j\\
        n-i & \hbox{ if }j<i
        \end{cases}
\end{equation*}
  
If $C_n$ is a cycle on $n$ vertices $\{x_0,\ldots,x_{n-1}\}$, then \\ 
if $n$ is even,
\begin{equation*}
 B(x_i,x_0) =  \begin{cases}
\frac{n-1-2i}{2} & \hbox{ if } 1\leq i<n/2\\
       0 & \hbox{ if }i=n/2
        \end{cases}
\end{equation*}
if $n$ is odd, 
$$ B(x_i,x_0) = 
\frac{n-1-2i}{2} \;\; \hbox{  if  } 1\leq i\leq \frac{n-1}{2}$$
By symmetry,
$$B(x_i,x_0)=B(x_{n-1},x_0)$$\\

For a star  $S_n$ with central vertex  $x_0$, 

 $$B(x_i,x_0) = 0\;B(x_0,x_i) = n-2$$ $$B(x_i,x_j) = 0\hbox{ for } i,j\neq 0$$\\
 
For a wheel  $W_n, n>5$ with central vertex  $x_0$, 

 $$B(x_i,x_0) = 0,\;B(x_0,x_i) =n-5 $$ $$B(x_i,x_{i\pm 1} ) = 1/2,\; B(x_i,x_{i\pm j} ) = 0\hbox{ for }j\geq 2$$
\begin{theorem}
Let $G$  be a graph and $x_i\in V(G)$, then\\$$B_G(x_i)=\frac{1}{2}\sum_{j\neq i}B_G(x_i,x_j)$$
\end{theorem}
\begin{definition}
Let $G$ be a graph and $x\in V(G)$. Let $S$,   $T$ be two disjoint subsets of $V(G)$, then the betweenness centrality of $x$ induced by $S$ and $T$ denoted by $B(x,S,T)$  may be defined as\\ 
$$B(x,S,T)=B(x,S)+B(x,T)$$
\end{definition} 
\begin{definition}
Let $G$ be a graph and $x\in V(G)$. Let $S$,          $T$ be two disjoint subsets of $V(G)$ where $s(\neq x)\in S$ and $t(\neq x)\in T$, then the betweenness centrality of $x$ induced by $S$ and $T$ one against the other denoted by $B(x,S | T)$  may be defined as\\ 
$$B(x,S|T)=\sum_{s\in S,\,t\in T}\frac{\sigma_{st}(x)}{\sigma_{st}}$$ 
\end{definition} 
\section{Betweenness centrality of a vertex in GP(n,2)}
Let us consider the betweenness centrality of $GP(n,2)$, where the vertices lie on two vertex transitive subgraphs namely the outer cycle generated by  $U=\{u_0,u_1,\ldots,u_{n-1}\}$ and the inner cycle generated by $V=\{v_0,v_1,\ldots v_{n-1}\}$.
\subsection{Betweenness centrality of an outer vertex in GP(n,2)}
\begin{theorem}
The betweenness centrality of an outer vertex $u$ in $GP(n,2)$ is given by\\
$$B(u)~= 
\begin{dcases}\frac{1}{4}(5n+1) & \hbox{ for }  n=13,17,21,\ldots\\ 
        \frac{15n^2+32n-79}{12(n+1)} &\hbox{ for  }n=15,19,23,\ldots\\
\frac{1}{4}(5n+14) &\hbox{ for even }n,\,  n\geq 12
\end{dcases}$$
\end{theorem}
\begin{proof}
The betweenness centrality of a vertex in $G$  is the sum  of the  betweenness centralities induced by  $U$, $V$ and $U$ $Vs$ $V$ determined in lemma \ref{lemb2} -\ref{lemb4}.
\end{proof}
\begin{lemma}\label{lemb1}
For any vertex $u_0$ in the outer cycle of $GP(n,2)$, $n\geq 6$,   there are 10 pairs of outer vertices, and for each pair there is a geodesic lying on the outer cycle  with $u_0$ as an internal vertex. More over, these pairs contribute the value 6.5 for its betweenness centrality.
\end{lemma}
\begin{proof}
Consider $GP(n,2)$, $n\geq 6$. Now for any vertex $u_0\in U$, followed by lemma \ref{l1}, it can be easily seen that there  exists a geodesic joining $u_{-1}$ to $u_r$ and $u_{-2}$ to $u_s$ where $r=1,2,3,4$ and $s=2,3$ lying entirely on the outer cycle,  contributing  1, 1, 1/2, 2/3 and 1/2, 1/3 respectively to the betweenness centrality of $u_0$. Hence by the symmetry of metric, there are 10 pairs of vertices with total contribution  13/2.
\end{proof}
\begin{lemma}\label{lemb2}
In $GP(n,2)$, the betweenness centrality of an outer vertex $u_0$ induced by the outer cycle is given by

$$B(u_0,U)~= 
\begin{dcases}\frac{1}{4} (n+13) & \hbox{ for } n=13,17,21,\ldots\\ 
 \frac{1}{12}(3n+41) &\hbox{ for  }n=15,19,23,\ldots\\
  \frac{1}{4} (n+14) &\hbox{ for  even }n,\,n\geq 12 
\end{dcases}$$
\end{lemma}

\begin{proof}
Consider an outer vertex $u_0\in U$ in $GP(n,2)$ for $n\geq 12$. First we take all possible $U$-$U$ pairs of outer vertices and find their contributions to $B(u_0)$. By lemma \ref{lemb1} the outer cycle contains 10 geodesics passing through $u_0$ contributing 13/2. When $n=2k+1$, for even $k$, $k\geq 8$, the pair $(u_{-1},u_r)$, $r=6,8,\ldots,k-2$ contributes 1/2 for each $r$ and when $k$ is incremented, there is one more pair $(u_{-1},u_{k-1})$ with contribution 1/3. Therefore, by symmetry, the outer pairs contribute the sum $\frac{1}{2} (k+7)$ when $k$ is even and $\frac{1}{6}(3k+22)$ when $k$ is odd.  When $n=2k$, for  even $k$, $k\geq 8$, the pair $(u_{-1},u_r)$, $r=6,8,\ldots,k-2$ contributes 1/2 for each $r$ and when $k$ is incremented, there is one more pair $(u_{-1},u_{k-1})$ with contribution 1/4. Hence by symmetry, it can be seen that outer pairs contribute  $\frac{1}{2} (k+7)$ for $k$, even or odd. 
\end{proof}
\begin{lemma}\label{lemb3}
In $GP(n,2)$, the betweenness centrality of an outer vertex $u_0$ induced by the inner cycle is given by
$$B(u_0,V)~= 
\begin{dcases} \frac{1}{2}(n-5) & \hbox{ for } n=13,17,21,\ldots\\ 
 \frac{n^2-2n-19}{2(n+1)} &\hbox{ for  }n=15,19,23,\ldots\\
   n/2 &\hbox{ for  even }n,\,n\geq 12 
\end{dcases}$$
\end{lemma}
\begin{proof}
 Consider the possible $V$-$V$ pairs of inner vertices. When $n=2k+1$ for even  $k$, $k\geq 6$, the pair $(v_0,v_r)$  for $r=1,3,\ldots,k-3$ contributes the betweenness centrality $1/t$ where $t=(r+1)/2$. By the symmetry of the metric, there are 2$t$ similar pairs  with a total contribution 2 for each $r$. Since there are $(k-2)/2$ of such $(v_0,v_r)$ pairs, the total contribution of $V$-$V$ pairs is $k-2$. When $k$ is incremented, there is one new pair $(v_0,v_r)$ where $r={k-2}         $  with contribution $1/t$ where $t=(k+1)/2$ and there are $k-1$ similar pairs. Therefore, by symmetry, inner pairs contribute the sum $\frac{k^2-5}{k+1}$. Consider the case $ n=2k$, for even $k,\; k\geq 6$ the inner pairs $(v_0,v_r)$ for $r=1,3,\ldots,k-1$ contributes $1/t$ where $t=(r+1)/2$. Because of symmetry, each pair $(v_0,v_r)$ belongs to a set of $2t$ similar pairs giving 2 and these sets contributes  a total $k$ to the centrality of $u_0$. When $k$ is incremented, the leading pair $(v_0,v_r)$ for $r=k$ gives $1/t$ where $t=(k+1)/2$ and so get the total contribution as $k$. 
\end{proof}
\begin{lemma}\label{lemb4}
In $GP(n,2)$, the betweenness centrality of an outer vertex $u_0$ induced by the vertices of outer Vs inner cycle is given by
$$B(u_0,U|V)~= 
\begin{dcases} \frac{1}{2}(n-1) & \hbox{ for odd } n,\,n\geq 13\\ 
   n/2 &\hbox{ for  even }n,\,n\geq 12 
\end{dcases}$$
\end{lemma}
\begin{proof}
Consider the possible $U$-$V$ pairs. When $n=2k+1$, for even $k$, $k\geq6$, the pair $(u_1,v_{-r})$ for $r=0,2,4,\dots,k-2$ contributes 1,  When $k$ is incremented, the leading pair $(u_1,v_{-r})$ makes the contribution 1/2 for $r=k-1$. Therefore  by symmetry, in either case the total contribution can be found as $k$. Similar argument is there for $n=2k$. 
\end{proof}
\subsection{Betweenness centrality of an inner vertex in GP(n,2)}
\begin{theorem}
The betweenness centrality of an inner vertex $v$ in $GP(n,2)$  is given by\\
$$B(v)~= 
\begin{dcases}\frac{1}{4}(n^2-n-26) & \hbox{ for } n=13,17,21,\ldots\\ 
 \frac{3n^3-83n+16}{12(n+1)} &\hbox{ for  }n=15,19,23,\ldots\\
\frac{1}{4}(n+5)(n-6) &\hbox{ for even }n,\,  n\geq 12
\end{dcases}$$
\end{theorem}
\begin{proof}
The betweenness centrality of $v\in V$  is the sum  of its betweenness centralities induced by the subsets $U$, $V$ and $U$ $Vs$ $V$ in $G$ determined in lemma \ref{lemb5} -\ref{lemb7}.
\end{proof}
\begin{lemma}\label{lemb5}
In $GP(n,2)$, the betweenness centrality of an inner vertex $v_0$ induced by the outer cycle is given by
$$B(v_0,U)~= 
\begin{dcases}\frac{1}{16}(n^2+6n-127) & \hbox{ for } n=13,17,21,\ldots\\ 
 \frac{1}{48}(3n^2+18n-377) &\hbox{ for  }n=15,19,23,\ldots\\
  \frac{1}{16}(n^2+6n-128) &\hbox{ for  even }n,\,n\geq 14 
\end{dcases}$$
\end{lemma}
\begin{proof}
 Consider all outer pairs $(u_i,u_j)$ such that $v_0$ lies on atleast one geodesic joining them. Let $d=d(u_i,u_j)$, then clearly $4\leq d\leq D$, where $D=diam(G)$. Let $B_d(v_0)$ denotes the total contribution of those pairs at a distance $d$ towards the betweenness centrality $B(v_0)$ of $v_0$. Consider $n=2k+1$ where $k=2l,\,l\geq 3$. Now for $d=4$, there exists 3  pairs of 1/2  and for $d=5$, there exists 6 pairs of 1/3 and 4 pairs of 1.  For $6 \leq d\leq D$, there exist $2d-4$ pairs of 1/2 and $d-1$ pairs of 1 giving the sum $2d-3$. Therefore,
 $$B(v_0,U)=\sum_{d=4}^DB_d=15/2+\sum_{d=6}^D(2d-3)=\frac{1}{16}(n^2+6n-127)$$
In the case $n=2k+1$, when $k$ is incremented from odd to even, the leading diametric pairs $(u_i,u_j)$ i.e, vertices at a distance $D=(n+9)/4$, contains 2 pairs of 2/3 and $3(n-3)/4$ pairs of 1/3 and hence there is a contribution of an extra sum $(3n+7)/12$. Therefore,\\
$$B(v_0,U)=\frac{1}{16}(n^2+2n-135)+(3n+7)/12= \frac{1\label{l1}}{48}(3n^2+18n-377) $$
Consider the case $n=2k$ where $k=2l,\,l\geq 3$, then $D=(n+8)/4$ and when $d=D$, there are $(3n-4)/4$ pairs of 1/2 and a single pair of 1 giving the sum $(3n+4)/8$. Therefore,\\
$$B(v_0,U)=15/2+\sum_{d=6}^{D-1}(2d-3)+(3n+4)/8=\frac{1}{16}(n^2+6n-128)$$ 
Consider the case $n=2k$ where $k=2l+1,\,l\geq 3$, then $D=(n+10)/4$ and when $d=D$, there are $(n-2)/2$ pairs of 1/4 and one diametric pair of 2/4\label{l1} giving the sum $(n+2)/8$. Therefore,\\
$$B(v_0,U)=15/2+\sum_{d=6}^{D-1}(2d-3)+(n+2)/8=\frac{1}{16}(n^2+6n-128)$$ 
\end{proof}
\begin{lemma}\label{lemb6}
In $GP(n,2)$, the betweenness centrality of an inner vertex $v_0$ induced by the inner cycle is given by
$$B(v_0,V)~= 
\begin{dcases}\frac{1}{16}(n^2-6n+21) & \hbox{ for } n=13,17,21,\ldots\\ 
 \frac{n^3-5n^2+3n+137}{16(n+1)} &\hbox{ for  }n=15,19,23,\ldots\\
  \frac{1}{16}(n-2)(n-4) &\hbox{ for  even }n,\,n\geq 12 
\end{dcases}$$
\end{lemma}
\begin{proof}
 We consider all inner pairs $(v_i,v_j)$ such that $v_0$ lies on their shortest paths. Le\label{l1}t $d=d(v_i,v_j)$ and $D=diam(G)$, then clearly $2\leq d\leq D$. Let $B_d(v_0)$ denotes the total contribution of pairs at a distance $d$ towards the betweenness centrality $B(v_0)$ of $v_0$. Let $n=2k+1$ where $k=2l,\,l\geq 3$, then $D=(n+7)/4$. For $d=2$, there exists a  pair contributing 1  and for $2\leq d\leq D-1$ we have  $B_d(v_0)=2d-4$. Therefore,
 $$B(v_0,V)=\sum_{d=2}^{D-1}B_d=1+\sum_{d=3}^{D-1}(2d-4)=\frac{1}{16}(n^2-6n+21)$$

In the case $n=2k+1$ where $k=2l+1,Gplay\,l\geq 3$, we have $D=(n+9)/4$ and 

$$B(v_0,V)=\sum_{d=2}^{D-1}B_d=1+\sum_{d=3}^{D-2}(2d-4)+(D-4)+4(D-2)^{-1}=\frac{n^3-5n^2+3n+137}{16(n+1)}$$
\\
Consider $n=2k$ where $k=2l,\,l\geq 3$. 
There are two inner $k$-cycles. Since $v_0$ lies on a $k$-cycle, subscripted with even numbers, we need not consider the pair $(v_i,v_j)$ with odd subscripts $i$ and $j$.  The even subscripted pairs $(v_i,v_j)$ give the betweenness centrality $(k-2)^2/8$. Now for $j=2,4,6,$ etc, the pair $(v_{-1},v_j)$ gives $1/2,1/3,\ldots, 1/l$; $(v_{-3},v_j)$ gives $2/3,2/4,\ldots, 2/l$ and finally $(v_{-(k-3)},v_2)$ gives $(l-1)/l$. Considering vertices of these two inner $k$-cycles. $B_d(v_0)=d-3$ for $4\leq d\leq D$ where $D=k/2+2$. Therefore,
$$B(v_0,V)=\sum_{d=4}^DB_d+\frac{(k-2)^2}{8}=\frac{1}{16}(n-2)(n-4)$$
In the case  $n=2k$ where $k=2l+1,\,l\geq 3$, the vertices of the same cycle contribute $(k-1)(k-3)/8$ to $v_0$ and for the vertices of the different cycles $B_d(v_0)=d-3$ for $4\leq d\leq D-1$ and $B_D(v_0)=(k-1)/4$ where $D=(k+5)/2$. Now
$$B(v_0,V)=\sum_{d=4}^DB_d+\frac{(k-1)(k-3)}{8}=\frac{1}{16}(n-2)(n-4)$$
\end{proof}
\begin{lemma}\label{lemb7}
In $GP(n,2)$, the betweenness centrality of an inner vertex $v_0$ induced by the vertices of outer Vs inner cycle is given by
$$B(v_0,U|V)~= 
\begin{dcases}\frac{(n-1)^2}{8}& \hbox{ for } n=13,17,21,\ldots\\ 
 \frac{1}{8}(n^2-2n+5) &\hbox{ for  }n=15,19,23,\ldots\\
  \frac{n(n-2)}{8} &\hbox{ for  even }n,\,n\geq 12 
\end{dcases}$$
\end{lemma}
\begin{proof}
In $GP(n,2)$ where $n=2k+1$, $k=2l$, $l\geq 3$,
consider the possible $u$-$v$ pairs determining the value of $B(v_0)$. From $u_0$ there are $k/2$ geodesics passing through $v_0$ to either sides.  For an even index $i\leq k-2$, from  $u_i$ and $u_{i-1}$  there is an equal number of geodesics, i.e, $(k-i)/2$ passing through $v_0$  and therefore, by symmetry, there exists $k^2/2$, $u$-$v$ geodesics  through $v_0$.
Hence $$B(v_0,U|V)=\frac{(n-1)^2}{8}\;\hbox{ for } n=13,17,21,\ldots$$ 
In the case of odd $k$, i.e, $k=2l+1$, $l\geq 3$ there is one more geodesic of 1/2 from each outer vertex and  the expression $k^2/2$ turns to be $(k-1)^2/2+k$ Hence $$B(v_0,U|V)=\frac{1}{8}(n^2-2n+5)\;\hbox{ for  }n=15,19,23,\ldots$$ 
Consider the case $n=2k,\,k\geq 7$. From $u_0$ to any inner even vertex $v_j$, $(j\neq 0)$ there is a unique geodesic passing through $v_0$, and from $u_i$ to $v_j$ there are two if $v_i$ and $v_j$ are diametric pairs of the inner even cycle. Thus  $k(k-1)/2$  stands for $B(v_0)$. Hence
$$B(v_0,U|V)=\frac{n(n-2)}{8}\;\hbox{ for  even }n,\,n\geq 12$$ 
\end{proof}
\section{Conclusion}
Here we found the number of geodesics between two vertices and the betweenness centrality of $GP(n,k)$ for $k=2$.  The study of geodesics is extremely important in the context of interconnection networks and it has wide applications in routing , fault-tolerance, time delays and in the calculation of many  centrality measures. This work may be extended for any $k<n/2$.
\bibliographystyle{unsrt}
\bibliography{sunilbiblio}
\end{document}